\documentclass[12pt, twoside]{amsart}

\usepackage{amssymb, amsmath, amsthm, mathrsfs,multirow,}
\usepackage{graphicx, color, wrapfig,}
\usepackage{tikz}
\usetikzlibrary{calc}
\usepackage{hyperref}
\usepackage[mathscr]{euscript}
\usepackage{enumerate}
\usepackage{fullpage}
\graphicspath{{figures/}} % figures kept in subdirectory or folder
\usepackage{varwidth}
\usepackage{cancel} % $\cancel{x}$ to get a slash through x in math mode
\usepackage{hyperref} % for links within the document

\hypersetup{
    colorlinks,%
    citecolor=black,%
    filecolor=black,%
    linkcolor=black,%
    urlcolor=blue
}

\usepackage{todonotes}

\usepackage{lipsum}
\usepackage{floatrow}
\usepackage[latin1]{inputenc}

\newcommand{\mat}[1]{\left(\begin{matrix}#1\end{matrix} \right)} 

\theoremstyle{plain}

\theoremstyle{remark}

\theoremstyle{definition}

\newtheorem{theorem}{Theorem}[section]
\newtheorem{lemma}[theorem]{Lemma}
\newtheorem{corollary}[theorem]{Corollary}

\newtheorem*{theorem*}{Theorem}
\newtheorem*{lemma*}{Lemma}
\newtheorem*{exercise*}{Exercise}

\theoremstyle{definition} % text is not in italics
\newtheorem{definition}[theorem]{Definition}
\newtheorem{remark}[theorem]{Remark}
\newtheorem{proposition}[theorem]{Proposition}

\newtheorem{example}[theorem]{Example}
 
\newtheorem*{definition*}{Definition}
\newtheorem*{notation*}{Notation}
\newtheorem*{example*}{Example}
\newtheorem*{question*}{Question} % added for use in challenges

\title[Limit linear series  and tropical divisors on chains.]
{Limit linear series on chains of elliptic curves and tropical divisors on chains of loops}

\author[A. L\'opez]{Alberto L\'opez Mart\'in}
\address{Instituto de Matem\'atica Pura e Aplicada, Estrada Dona Castorina 110, 22460-320 Rio de Janeiro RJ, Brazil}
\email{alopez@impa.br}

\author[M. Teixidor]{Montserrat Teixidor {\textrm i} Bigas} 
\address{Department of Mathematics, Tufts University,503 Boston Avenue, Medford, MA 02155, USA}
\email{montserrat.teixidoribigas@tufts.edu}

\subjclass[2010]{Primary 14H60, 14T05 $\cdot$ Secondary 14H51}

\date{}
\begin{document}

%%%%%%%%%%%%%%%%%%%%%%%%%%%%%%%%%%%%%%%%%%%%%  ABSTRACT
\begin{abstract}
 We describe the space of Eisenbud-Harris limit linear series on a chain of elliptic curves and compare it with the theory of divisors on tropical chains. Either model allows to compute some invariants of Brill-Noether theory using combinatorial methods. We introduce effective limit linear series.
\end{abstract}

\maketitle
%%%%%%%%%%%%%%%%%%%%%%%%%%%%%%%%%%%%%%%%%%%%%  INTRODUCTION
\begin{section}*{Introduction}
Limit linear series were introduced by Eisenbud and Harris \cite{EH86} in the early eighties
 and have since been used extensively as a very effective tool in dealing with a variety of problems related to moduli spaces of curves and Jacobians. 
 This theory is applicable to any curve of compact type, which, by definition,  is a curve whose dual graph has no loops or, equivalently, whose Jacobian is compact.
Shortly after the introduction of limit linear series, 
 Welters pointed out in  \cite{W}  that chains of elliptic curves are very well behaved in terms of their limit linear series.
 Welters remarked that, even in positive characteristic, chains of elliptic curves provide straightforward proofs of the basic results in Brill-Noether theory
 (see \cite{CT} section 2 for a proof of the Gieseker Petri Theorem using these curves). 
Since then, chains of elliptic curves have been used mostly to tackle problems in higher rank Brill-Noether theory 
--- an account can be found in \cite{Clay} --- and more recently again on the classical theory (see, for instance, \cite{P}).

In  this paper we do the following
\begin{enumerate}[(1)]
\item We show that, for a chain of elliptic curves, the Brill-Noether locus  is reducible with components explicitly described and in correspondence with  fillings of a certain Young Tableaux.
\item We show that, for  a chain of loops, the tropical Brill-Noether locus  is reducible with components corresponding to fillings of a certain Young Tableaux.
\item We define effective limit linear series and show they are equivalent, in the refined case, to the traditional limit linear series.
\item We establish a comparison between the limit linear series point of view and the tropical point of view.
\end{enumerate}

Our results on (1)(see section 1) reduce many algebro-geometric problems to combinatorial questions and can be used in explicit computations.
 In fact, in collaboration with Chan and Pflueger ( \cite{CLPT} ), we used this approach to compute the genus of a Brill-noether locus when the dimension of this locus is one
 (see also  \cite{CLT} ).
 Among the potential applications are the computation of the Euler -Poincare characterisitic of the Brill-Noether locus. 
 Some of the applications may require to extend the description to the space of limit series themselves and not just its image in the Jacobian.
 We will be considering this extension in forthcoming work.
 
 The results in (2)  were known only in the case when the tropical Brill-Noether locus is finite  (see \cite{CDPR}).
  In section 2, we show that the result extends to arbitrary dimension $\rho$
  with points in the Brill-Noether locus  in the case $\rho =0$ being replaced by sub-chains of loops from the original chain for arbitrary $\rho$.
  
  The concept  of effective linear series (section 3) is  a variation of the concept of limit linear series in \cite{EH86}:
   instead of successively concentrating all of the degree in one component, 
we leave just enough of it behind so that a line bundle would have a section on each component but still the dimension of the space of sections 
on the chosen component is the dimension of the limit linear series.
It is based on the insight that limit linear series should not be thought of as made up of unrelated pieces 
  on each component of the reducible curve but should rather be considered as line bundles and sections defined globally. 
  This point of view is useful in a number of questions. It is currently being used to deal with the maximal rank conjecture (see \cite{LOTZ}).
  We expect it will find applications to a number of other questions related to  generation like  the  study of  kernels of  evaluation maps and their impact on syzygies. .

  Finally, the comparison between the tropical and the limit linear series approach shows the equivalence of the two methods
   and should facilitate more fruitful conversations among researchers coming from the two different backgrounds.
 Our strategy is to  show that the orders of vanishing at the nodes of the effective linear series 
 agree with the tropical orders of vanishing (see Remark \ref{Rem}). 
 The proof of the Brill-Noether theorem is based, in both cases, in the use of these orders of vanishing, so the  proofs in the two set ups run in parallel.
We  include both proofs  here presented in a way that  illustrates the similarities  between the two.

\end{section}
%%%%%%%%%%%%%%%%%%%%%%%%%%%%%%%%%%%%%%%%%%%%%  SECTION 1
\section{Limit linear series on chains of elliptic curves}\label{LLSdesc}

Limit linear series were introduced by Eisenbud and Harris  and they model the behavior of a linear series
 when an irreducible curve degenerates to a reducible nodal curve of compact type. 
Assume that  we have a one dimensional family of curves in which all fibers but one are irreducible while the special fiber is a nodal curve.
Under good conditions for the total space of the family, each of the components of the special fiber corresponds to a divisor on the total space. 
Given  a line bundle on the family, one can modify it by tensoring with line bundles  with support on the central fiber.
This will leave the restriction of the line bundle to the generic fiber unchanged while redistributing the degrees among the components of the central fiber.
Limit linear series isolate the data of the restrictions of these line bundles when the degree (and therefore the space of sections) is concentrated on a single component:

 \begin{definition}\label{lls} Let $C$ be a nodal curve of compact type (that is, whose dual graph has no loops). 
 A limit linear series of degree $d$ and (projective) dimension $r$ on $C$ consists of the data of a line bundle $L_i$ 
 of degree $d$ on each component $C_i$ of $C$ and a space of sections $V_i$ of $H^0(C_i, L_i)$ of dimension $r+1$ 
 satisfying the following condition: 
  Assume that  $P_{j_1(\alpha)} \in C_{j_1(\alpha)}$ is identified to  $P_{j_2(\alpha)} \in C_{j_2(\alpha)}$ to form a node $P_{\alpha}$  of $C$.
  Consider  the $r+1$ distinct orders of vanishing $u_0(j_1)>\dots>u_r(j_1)$ at $P_{j_1}$ of the sections in  $V_{j_1}$
  and the $r+1$ distinct orders of vanishing $u_0(j_2)>\dots>u_r(j_2)$ at $P_{j_2}$ of the sections in  $V_{j_2}$. Then, $u_t(j_1)+u_{r-t}(j_2) \ge d$, $t=0, \dots, r$.
  The series is called refined if $u_t(j_1)+u_{r-t}(j_2) = d$, $t=0, \dots, r$ for all nodes.
  \end{definition}
As we mentioned above, the definition of limit linear series  is inspired by what would appear as the limit of a linear series when the degree of the limit line bundle is 
concentrated successively in the various components. 
The relationship among the vanishing on the two components gluing at a node reflect the way these  limit bundles are related to each other:
With the notation above, $C-\{ P_{\alpha}\}$ is the union of the two connected components
$X_{j_1},  X_{j_2}$ that contain $C_{j_1}-\{ P_{j_1}\} , C_{j_2}-\{ P_{j_2}\}$, respectively. 
Assume that ${\mathcal L}_i$ is the line bundle on the family whose restriction to $C_{j_i}$ has degree $d$ and whose restriction to any other component has degree zero. 
Then, one can check that  ${\mathcal L}_2={\mathcal L}_1(-dX_2)$. 
If a section $\sigma $ of ${\mathcal L}_1$ vanishes with order $k$ on $X_{j_2}$ and $t $ is an equation of $C$ on the family, 
then $t^{d-k}\sigma$ is a section of ${\mathcal L}_{j_2}$ and vanishes on $X_{j_1}$ to order $d-k$.  Therefore, it vanishes at $P_{j_1}$ to order at least $d-k$.
This is what the relationship among the orders of vanishing reflects.

 \begin{definition}\label{cce} Let $C_1,\dots,C_g$ be elliptic curves, $P_i, Q_i \in C_i $ such that $P_i-Q_i$ is not a torsion point of $C_i$.
  Glue $Q_i$ to $P_{i+1},i = 1,\dots,g-1$ to form  a node. The resulting curve will be called  a \emph{general chain of elliptic curves} (of genus $g$). \end{definition}
     
A  general chain of elliptic curves is Brill-Noether general. 
We describe here the image in the jacobian of  the space of limit linear series of a fixed degree and dimension on such a curve. 
We first need to understand what goes on on a single elliptic component.
 As the canonical (dualizing) sheaf on an elliptic curve is trivial, the Riemann-Roch Theorem on elliptic curves is particularly simple.
 
   %RR CORBA EL.LIPTICA
  \begin{lemma}[Riemann-Roch Theorem]\label{RR} Given a line bundle $L$ of degree $d$ on an elliptic curve, then
  \begin{itemize}
  \item If $d<0$, then $h^0(L)=0$.
  \item If $d>0$, then $h^0(L)=d$.
   \item If $d=0$, then $h^0(L)=1$  if $L=\mathcal{O}$ and $h^0(L)=0$ otherwise.
   \end{itemize}
   \end{lemma}

The next two lemmas have been used repeatedly in Brill-Noether questions for vector bundles (see, for instance, \cite{Tei05} Lemma 2.2 and \cite{Tei08b} Lemma 2.2) 
and exploit Lemma \ref{RR}, especially the third point.
  We include them here for the convenience of the reader.

\begin{lemma}\label{cotsupanul} Consider an elliptic curve $C$ and two points $P, Q\in C$ such that $P-Q$ is not a torsion point in the group structure of $C$.
Let $L$ be a line bundle of degree $d$ on $C$ and $V$ a space of sections of $L$ of dimension $r+1\le d$. 
Let the  orders of vanishing of the sections of $V$ at $P$  and $Q$ be, respectively,
\[u_0(P)>\dots>u_r(P)\mbox{ and }u_0(Q)>\dots >u_r(Q).\]
 Then $u_t(P)+u_{r-t}(Q)\le  d$ and  $u_t(P)+u_{r-t}(Q)=  d$ for, at most, one value $t$.
\end{lemma} 

\begin{proof} Note that, by definition, the dimension of the space of sections of $V$ that vanish to order $u_t$ at $P$, 
$\dim  V(-u_t(P)P)=t+1$ and $\dim  V(-u_{r-t}(Q)Q)=r-t+1$. Therefore, 
$$\dim  [V(-u_t(P)P)\cap  V(-u_{r-t}(Q)Q)]\ge t+1+r+1-t-\dim V=1.$$
There is then a section vanishing to order $u_t(P)$ at $P$ and $u_{r-t}(Q)$ at $Q$. As the degree of the line bundle is $d$, this requires that  $u_t(P)+u_{r-t}(Q)\le d$.
Moreover, if $u_t(P)+u_{r-t}(Q)= d$, then $L=\mathcal {O}(u_t(P)P+u_{r-t}(Q)Q)$. If there were another value $t'$ such that  $u_{t'}(P)+u_{r-t'}(Q)= d$,
 then  also $L=\mathcal {O}(u_{t'}(P)P+u_{r-t'}(Q)Q)$. This implies that 
 $u_t(P)-u_{t'}(P)=u_{r-t'}(Q)- u_{r-t}(Q)$. Hence, $(u_t(P)-u_{t'}(P))(P-Q) \equiv 0$ 
contradicting the assumptions about the generality of $P, Q$.
\end{proof}

%EXISTENCIA FIBRATS DE LINEA AMB ANUL.LAMENTS DONATS
\begin{lemma} \label{exfibanul} Given an elliptic curve $C$, two points $P, Q\in C$ such that $P-Q$ is not a torsion point, 
 and integers $d\ge u_0>\dots>u_r\ge 0$.
 \begin{enumerate}
\item There exists then a one-dimensional family of line bundles $L$  of degree $d$ on $C$ and, for each of them, a unique space of sections  $V$ of $L$ of dimension $r+1$
with orders of vanishing at $P$  and $Q$ being, respectively,
 \[ d-u_r, \dots, d-u_0\mbox{ and } u_0-1,\dots, u_r-1\]
 if and only if  $u_r>0$.
\item There exists a unique  line bundle $L$  of degree $d$ on $C$ and  space of sections  $V$ of $L$ of dimension $r+1$
with orders of vanishing at $P$  and $Q$ being, respectively,
 \[ d-u_r, \dots, d-u_0\mbox{ and } u_0-1,\dots, u_{t_0-1}-1, u_{t_0}, u_{t_0+1}-1,\dots  u_r-1\]
  if and only if  $u_{t_0}+1<u_{t_0-1}$ when $t_0\not= 0$ and  $u_r>0$ when $t_0\not= r$.
\end{enumerate}
\end{lemma} 

\begin{proof}\ \\

\begin{enumerate}
\item The condition is necessary by definition, as an order of vanishing must be at least 0.

Conversely, choose an arbitrary line bundle $L$ of degree $d$ on $C$. Then, $h^0(L(-(d-u_t)P-(u_t-1)Q))=1$. 
Therefore, there is a unique section $s_t$ of $L$ that vanishes at $P$ to order at least $d-u_t$ and at $Q$ to order at least $u_t-1$.
Moreover, unless $L=\mathcal {O}( (d-u_t)P+u_tQ) $ or $L=\mathcal {O}( (d-u_t+1)P+(u_t-1)Q) $,
 this section vanishes to order precisely  $d-u_t$ at $P$ and $u_t-1$ at $Q$.
 For a given value of $t$, the two exceptions listed completely determine the line bundle.
 There is a finite number of possible values of $t$ and therefore a finite number of possibly exceptional  line bundles $L$. 
If the identity occurred for two different values $t, t'$ and the same line bundle, $P-Q$ would be a torsion point against our assumptions.

 Assume that we are not in one of the exceptional situations. 
  Then, the sections $s_i$ have different orders of vanishing at $P$ and are therefore independent. Define $V$ the space generated by these sections.
  Then the pair $(L, V)$ is completely determined by these conditions.
 
  \item  The condition imposed on the $u_t$ is equivalent to saying that the numbers $u_0-1,\dots, u_{t_0-1}-1, u_{t_0}, u_{t_0+1}-1,\dots,  u_r-1$ are all different and non-negative. 
  From the proof of Lemma \ref{cotsupanul}, the only line bundle for which these orders of vanishing are possible is  $L=\mathcal {O}( (d-u_{t_0})P+u_{t_0}Q) $.
  This line bundle has a space of sections with the given vanishing if and only if one can find independent sections $s_t, 0\le t\le r$ vanishing at $P, Q$ with precisely the  given orders.
  In particular, this requires that 
\begin{align*}
&h^0( L(-(d-u_t)P-(u_t-1)Q))\ge 1, t\not= t_0, \mbox{ and }\\ &h^0( L(-(d-u_{t_0})P-u_{t_0}Q))\ge 1.
\end{align*}

  From the definition of $L$, these inequalities are in fact equalities. Therefore, we can choose unique sections of $L$ that vanish at $P, Q$ with the given orders.
  From the generality of the pair of points $P, Q$, only one section of the line bundle vanishes at $P, Q$ with orders adding up to $d$. 
  Hence, the order of vanishing at the two points cannot be larger than what is specified. As the orders of vanishing at one of the points are all different, 
  the sections are independent.
\end{enumerate}\end{proof}

Denote by $\rho (g, d,r)$ the Brill-Noether number that gives the expected dimension of the locus of line bundles of degree $d$ which have at least $r+1=k$
independent sections.
For simplicity of notation, we write $\bar k=g-1-d+k$ the dimension of the adjoint linear series to a linear series of degree $d$ and dimension $k$.
With this notation, $\rho =g-k\bar k$.
Denote by $c(k, \bar k)$ the number of rectangular standard Young tableaux with $k=r+1$ columns numbered $0,\dots, r$ and
 $\bar k =g-d+r$ rows numbered $1,\dots, g-d+r$.
 From the hook length formula, this number is given by (assuming $\bar k\le k$)
\[  c(k, \bar k)=\frac{(k\bar k)!}{(k+\bar k-1)((k+\bar k-2))^2\dots (k)^{\bar k} ((k-1))^{\bar k}\dots (\bar k))^{\bar k }(\bar k-1))^{\bar k -1}\dots (2)^21}.\]

%TEOREMA SERIES LIMIT
\begin{theorem}\label{teorserlim} The image in the Jacobian of the scheme of limit linear series of degree $d$ and dimension $k$ on a general chain of elliptic curves 
is reducible with 
\[ {g\choose {\rho}} c(k, \bar k)\]
components corresponding to the $c(k, \bar k)$ fillings of the $k\times \bar k$ Young diagram with $g-\rho=k\bar k$ numbers from the set $1,2, \dots, g$. 
Each of these components is birationally equivalent to a product of $\rho$ of the elliptic curves among the 
irreducible components in $C$ (the ones whose indices do not appear in the corresponding tableau).
\end{theorem}
\begin{proof} (Compare, for instance, with the proof of \cite[Thm 1.1]{Tei04} or the proof of \cite[Thm 1.1]{Tei08b}).
The orders of vanishing at $P_1$ (resp. $Q_g$) of an $r+1$-dimensional space of sections of a line bundle of degree $d$ on the curve $C_1$ 
(resp. $C_g$) are at least $(r, r-1, \dots, 0)$. From the definition of limit linear series, the sum of the orders of 
vanishing at $Q_i, P_{i+1}, i=1,\dots , g-1$ of the sections of a linear series is at least $(r+1)d$. Therefore, the sum of all the vanishing at the points 
$P_i, Q_i, i=1,\dots,g$ is at least $(g-1)(r+1)d+r(r+1)$. 

On the other hand, from Lemmas \ref{cotsupanul} and \ref{exfibanul}, the sum of the orders of vanishing at $P_i, Q_i$ is at most $(r+1)(d-1)$
for a general choice of line bundle and $(r+1)(d-1)+1$
if the line bundle in this component is of the form ${\mathcal O}((d-u_{t(i)})P_i+u_{t(i)}Q_i)$. 
In the latter case, the line bundle is completely determined by the vanishing at  $P_i$ and the choice of one index $t(i)$ which must satisfy  the condition   $u_{t(i)}+1<u_{t(i)-1}$.
Write $\alpha$ for the number of components where the line bundle is generic.
The sum of the vanishing orders at all $P_i, Q_i$ is at most $g(r+1)(d-1)+g-\alpha$.
Putting together the two inequalities, we obtain 
$$ (g-1)(r+1)d+r(r+1)\le g(r+1)(d-1)+g-\alpha,$$
which can be written as 
\[  \alpha \le g-(r+1)(g-d+r)=g-(r+1)\bar k =\rho.\]
The line bundles on the elliptic curves $C_i$ can take a finite number of values when  they have been chosen to be special  and can move on the (one-dimensional) Jacobian of 
$C_i$ otherwise. So $\alpha $ gives a bound for the dimension of the space of limit series. This shows that the dimension of any component 
of the scheme of limit linear series  is at most $\rho$.
 On the other hand, it is known from the construction of a space of limit linear series that every component has dimension at least $\rho$. 
 Therefore, every component has dimension precisely $\rho$.
 Note also that  a component of dimension $\rho $ corresponds to the choice of $\rho$ components  $C_i$ 
of $C$ where the line bundles $L_i$ are free to vary and the choice on each of the remaining components $C_j$ of an order of vanishing at $P_j$ satisfying the condition 
in Lemma \ref{exfibanul} (ii).

Let us see how this choice can be carried out. A component of dimension $\rho$ 
of the space of limit linear series, corresponds to  the data above so that all the inequalities are equalities.
We will denote by $u_0(i)>\dots>u_r(i)$  the orders of  vanishing of the sections at $Q_i$. 
As the inequalities are equalities, the orders of vanishing at $P_i$ must be 
  $(d-u_r(i-1),\dots, d-u_0(i-1))$.
  
When a line bundle is chosen generically,  Lemma \ref{exfibanul} (i) states that each vanishing order at $Q_i$ is one less than at $Q_{i-1}$, while for a special line bundle one of the vanishing orders stays the same while the rest decrease in one unit, therefore at a generic point of a component of the set of limit linear series:
\begin{enumerate}[(a)]
\item  The vanishing at $P_1, Q_g$ is $(r, r-1,\dots , 0)$. 
\item On $\rho$ of the components $C_i$, the line bundle is  generic  and then 
\[  (u_0(i),\dots, u_r(i))=(u_0(i-1)-1,\dots, u_r(i-1)-1).\] 
\item On $g-\rho$ of the components $C_i$, there is a $t(i)$ with $u_{t(i)}(i-1)+1<u_{t(i)-1}(i-1)$, the line bundle is  of the form  
$\mathcal {O}((d- u_{t(i)}(i-1))P_i+u_{t(i)}(i-1)Q_i) $    and then 
\[  (u_0(i),\dots, u_r(i))=(u_0(i-1)-1,\dots,u_{t(i) -1}(i-1)-1, u_{t(i) }(i-1), u_{t(i) +1}(i-1)-1,\dots , u_r(i-1)-1).\] 
\end{enumerate}
In keeping with (a) for $P_1$ and our other conventions, we write 
\[  (u_0(0),\dots, u_r(0))=(d, d-1,\dots , d-r).\] 
From this description, we can compute the orders of vanishing at any point $Q_i$ in terms of the values of the $t(j)$ for $j\le i$ as follows:
condition (a) says that $( d-u_r(0), \dots, d-u_0(0))=(r, r-1,\dots,0)$. Hence, $u_s(0)=d-s$. 
As the vanishing $u_s$ goes down by a unit on each component $u_s(i)=u_s(i-1)-1$
except if $t(i)=s$, writing $\delta_{a,b}=0, a\not=b,  \delta_{a,b}=1, a=b$, $\beta_{i,s}=\sum _{\{ j\le i\ :\ L_j \text{ special} \} }\delta_{t(j),s}$, 
\begin{equation}\label{anulserlim} u_s(i) =d-s-i+\beta_{i,s}.\end{equation}

Recall that  $s$ is a suitable candidate for $t(i)$ if  and only if $u_s(i-1)+1< u_{s-1}(i-1)$. From the formula we just obtained, this is equivalent to 
 \begin{equation}\label{condY}\beta_{i-1,s}<\beta_{i-1,s-1}. \end{equation}

We now place the $g-\rho =k\bar k$ indices corresponding to the curves with special line bundle in a $k\times \bar k $  Young tableau. 
 An index $i$ is placed in the first empty spot  in column $t(i), 0\le t(i)\le k-1=r$. So it ends up in row $\beta_{i,t(i)}$, 
 Our definition guarantees that the indices increase going down the columns,   Equation (\ref{condY}) guarantees that they increase moving along the rows to the right.
It remains to check that each column has heigth $\bar k$ and therefore the indices fill the tableau.
 This can be shown as follows: $u_j(0)=d-j, u_j(g)=r-j$,  $u_j(i)=u_j(i-1)-1$ if $i$ is not in column $j$ while $u_j(i)=u_j(i-1)$ if $i$ is in column $j$.
 Therefore, there are $g-[(d-i)-(r-i)]=\bar k$ indices in column $j$ ensuring that $1\le t(i)\le \bar k$.
 
 Note that in this case, the line bundle of the limit linear series on $C_i$ is
   \begin{equation}\label{fibrat} L_i= \mathcal {O}(( t(i)+i-\beta_{i,t(i)})P_{i}+(d- t(i)-i+\beta_{i,t(i)})Q_{i}). \end{equation} \end{proof}
   
In joint work with M. Chan and N. Pflueger \cite{CLPT}, we give a full description of the scheme $G^{k,\alpha,\beta}_d(X,p,q)$ 
of limit linear series of degree $d$ and dimension $k$ on a general chain of elliptic curves with prescribed ramification $\alpha$, respectively $\beta$ at a general point $p$,
 respectively $q$ when $\rho=1$. In the special case when  $\rho=1$, Theorem\ref{teorserlim} follows from the result in \cite{CLPT}.   

\begin{proposition} The components corresponding to two different Young tableaux intersect in the Jacobian 
if and only if  the indices that appear in both appear in boxes $(t_i, m_i), (t_j, m_j)$
with $t_i- m_i=t_j-m_j$. The dimension of the intersection of two such components equals the number of indices that do not appear in either tableau. 
\end{proposition}
\begin{proof} The correspondence between components of the locus  of limit linear series and Young tableaux is defined so  that 
on the elliptic curves $C_i$   whose indices do not appear in the Young tableau, the line bundle is free to vary.
For a component $l_0$ whose index  appears in the Young tableau in column $t_0$ and row $x_0$, the line bundle is completely determined as given 
in equation (\ref{fibrat}).  If the index appears in two tableaux that intersect, the line bundle determined by the tableaux must be the same. 
This is equivalent to the condition $t_i- m_i=t_j-m_j$. Conversely, if these conditions are satisfied whenever an index appears in both tableaux, 
then the line bundle determined by each of the tableaux on that component is the same. For components that appear in only one of the tableaux,
 the line bundle is determined by the position of the index on that tableau while line bundles on components whose indices  do not appear in either tableau
 are free to vary and therefore contribute 1 to the dimension of the intersection.
\end{proof}

%CAS TROPICAL
\section{The tropical case}\label{sectrop}  
 
In this section we look at the tropical proof of Brill-Noether presented in \cite{CDPR}
and make use of  their techniques to give a description of the Brill-Noether locus for tropical chains of loops in terms of Young Tableaux (see Theorem \ref{espaiBNTrop} ).
This presentation extends the results of \cite{CDPR} Theorem 1.4 that deals with the case of Brill Noether number zero.

\begin{definition} A tropical curve is a metric graph. 
 The free abelian group on the points of a tropical curve $\Gamma$ is called the set of divisors $Div(  \Gamma )$.
  A function on the graph is a continuous piecewise linear function  whose slope on each piece of the subdivision is integral.
 One can associate to such a function $\psi$ the divisor that at each point of $\Gamma$ has weight
  the sum of the incoming slopes on the edges of $\Gamma$ to which the point belongs.  
 Two divisors are said to be equivalent
 (written as $\equiv$) if they differ in the divisor of a piecewise linear function.
 The group of equivalence classes of divisors is called the Picard group of $\Gamma$ and is graded by degree. 
 \end{definition}

 \begin{definition}\label{ctr} Let $\mathscr{L}_1,\dots, \mathscr{L}_g, \mathscr{M}_1,\dots, \mathscr{M}_g$ be segments of (real) length
  $l_1,\dots, l_g, m_1,\dots m_g$. Identify the two ends of  $\mathscr{L}_i$ with the two ends of $\mathscr{M}_i $ to give points $Q_{i-1}, Q'_i$.
  Then identify $Q'_i$ with $Q_{i}$ to form a connected chain with $g$ loops. The resulting tropical curve will be called a \emph{chain of loops} (of genus $g$). 
  The chain is said to be general if  the lengths  are generic 
  (which from the point of view of Brill-Noether theory means that their quotients do not equal the quotient of two positive integers less than $2g-2$).\end{definition}

 We now state some facts about divisors on tropical curves that are similar to those on sections of line bundles on elliptic curves.
 
   \begin{lemma}[Tropical Riemann-Roch on a loop]\label{TRR} Given a divisor $D$ of degree $d$ on a single loop, two points $P, Q$ on the loop 
   and (when $d\ge 0$) a non-negative  integer $a\le d$, then
  \begin{enumerate}
  \item[(i)] If $d\le 0$, then $D$ is not equivalent to an effective divisor unless $D=0$.
  \item[(ii)] If $d>0, \ a<d$, then $D$ is equivalent to a divisor of the form $aP+(d-a-1)Q+R$ where $R$ is some point on the loop.
   \item[(iii)] If $P, Q$ are general, there is a unique equivalence class of divisors such that $D$ is equivalent to $aP+(d-a)Q$ and
    then $D$ is not equivalent to $bP+(d-b)Q$ for any integer $b\not= a$.
   \end{enumerate}
   \end{lemma}
\begin{proof} The proof of this result can be obtained from an easy computation (compare also with example 2.1 in \cite{CDPR}). 
As the divisor of a function has degree zero, the first point is clear. 

Part (ii) can be proved  by directly exhibiting   a piecewise linear function giving the equivalence.

For (iii), if $aP+(d-a)Q$ is equivalent to $bP+(d-b)Q$, then (if, say, $b\ge a$), then $(b-a)P$ is equivalent to $(b-a)Q$, which is not true for  any pair $a, b$ if $P, Q$ are general.
\end{proof}
  
Recall that a divisor $D$ on a tropical curve $\Gamma$ is said to be of rank $r$ if for every effective divisor $D'$ of degree $r$ on $\Gamma$, $D-D'$ 
is equivalent to an effective divisor.  

Consider a  chain $\Gamma$ of $g$ generic loops as in Definition \ref{ctr}.
  Denote by $\Gamma _i$ the $i^{th}$ loop and by $Q_{i-1}, Q_i$ the points of intersection with $\Gamma_{i-1}, \Gamma_{i+1}$ respectively.  
  Using Lemma \ref{TRR}, every divisor is equivalent to a divisor whose support outside a fixed $Q_j$ has at most one point on the interior of each $\Gamma_i$.
  
\begin{lemma} \label{ordnalTr} Let $\Gamma $ be a general chain  of $g$ loops and $D$ a divisor  on $\Gamma $ of rank at least  $r$. 
For every  $ k=1,\dots ,g, \ \  \ t=0,\dots , r ,\ \ \  i=0,\dots, g$
there exist indices 
 $ \epsilon_{k},   \ \epsilon_{k,t}\in \{ 0, 1 \}$,
 points in the loops  
 $x_{k}\in \Gamma_k-\{ Q_{k-1}\}, \    x_{k,t}\in \Gamma_k$ and integers 
$u_t(i), \  u_0(i)>\dots>u_r(i)\ge 0 $
such that  
\[ D\equiv tQ_{0}+\sum _{k\le i}\epsilon_{k,t}x_{k,t}+u_t(i)Q_i+\sum _{k>i}\epsilon_{k}x_{k}.\]
\end{lemma}
%DEMOSTRACIO LEMA TROPICAL
\begin{proof}  Using Lemma \ref{TRR} on each of the loops starting with the last one, one can successively bring most of the degree to $Q_{g-1}, Q_{g-2},\dots, Q_0$
leaving behind at most one point on each loop.
 So, $D$ is equivalent to 
 \[ uQ_{0} +\sum _{k\ge1}\epsilon_{k}x_{k},\]
 where $u$ is chosen as large as possible and $   x_{k}\in \Gamma_k-\{ Q_{k-1}\}$.
  As $D$ has rank at least $r$,   there is an effective divisor equivalent to $D-rQ_0$ and $u \ge r$.
  
Define
 \[ (u_0(0),u_1(0),\dots ,u_r(0))=(u, u-1,\dots, u-r).\]
 With this definition, 
 \[u_0(0)>\dots >u_r(0)\ge 0\] 
and for each $t=0,\dots, r$, $D$ is trivially equivalent to 
 $ tQ_{0}+u_t(0)Q_0+\sum _{k>0}\epsilon_{k}x_{k}$.
 
 Assume now that we found all of the $u_t(j), \epsilon_{j,t},    x_{j,t} ,\  j\le i-1, 0\le t\le r$ .
 Our goal is to find $\epsilon_{i,t}, \ x_{i,t}, u_t(i), \ 0\le t\le r$ such that
\[\ \ (*)  \ \ \  tQ_{0}+\sum _{k\le i}\epsilon_{k,t}x_{k,t}+u_t(i)Q_i+\sum _{k>i}\epsilon_{k}x_{k}\equiv D \]
and \[ u_0(i)>\dots >u_r(i)\ge 0. \] 
By the prior step, 
\[  D\equiv   tQ_{0}+\sum _{k\le i-1}\epsilon_{k,t}x_{k,t}+u_t(i-1)Q_{i-1}+\epsilon_{i}x_{i}+\sum _{k>i}\epsilon_{k}x_{k}.\]
 Using \ref{TRR} on $\Gamma_i$,  there exist  
 $\delta_{i,t}\in \{ 0, 1\}, \ x_{i,t}\in \Gamma _i-\{ Q_i\},  \alpha_t(i)\in \mathbb{Z}^+$ satisfying:
 \[ \ \ \   u_t(i-1)Q_{i-1}+\epsilon_{i}x_{i} \equiv     \delta_{i,t}x_{i,t}+\alpha_t(i)Q_i.     \]
We will choose $ \epsilon_{i,t}=\delta_{i,t}, u_t(i)=\alpha _t(i)$ except when $\epsilon _i=1$ for a $t=t_0$   $u_{t_0}(i-1)Q_{i-1}+x_i\equiv (u_{t_0}(i-1)+1)Q_{i}$ and $u_{t_0}(i-1)+1=u_{t_0-1}(i-1)$. 
In this case,  we choose $\epsilon_{i,t_0}=1, x_{i,t_0}=Q_i, u_{t_0}(i)= u_{t_0}(i-1)$.
With these choices, condition (*) is satisfied and it remains to check that $ u_0(i)>\dots >u_r(i) \ge 0$.

  \begin{enumerate}[(a)]
\item By the genericity of $\Gamma _i$, $u_t(i-1)Q_{i-1}$ is not equivalent to $ u_t(i-1)Q_i$ if $u_t(i-1)>0$. 
Therefore, if  $\epsilon _i=0$ and  $u_r(i-1)>0$, then $ \epsilon _{i,t}=1, t=0,\dots, r$ and
 \[ ( u_0(i),\dots ,u_r(i))= ( u_0(i-1)-1,\dots ,u_r(i-1)-1)\]
The inequalities among the vanishing orders  are then satisfied.
\item  If $\epsilon _i=0$ and  $u_r(i-1)=0$, then $u_r(i-1)Q_{i-1}+\epsilon _ix_i$ is identically zero. 
  Hence, $  \epsilon _{i,r}=0, \epsilon _{i,t}=1; \ \  t=0,\dots, r-1$
  \[ ( u_0(i),u_1(i), \dots ,u_{r-1}(i), u_r(i))= ( u_0(i-1)-1, u_1(i-1)-1,\dots ,u_{r-1}(i-1)-1,u_r(i-1)).  \ \]
  The inequalities are satisfied if $u_{r-1}(i-1)>1$. As $D-(r-1)Q_0-Q_{i}$ is equivalent to an effective divisor, this needs to be the case.
\item If $\epsilon _i=1$, $u_{t_0}(i-1)Q_{i-1}+x_i\equiv (u_{t_0}(i-1)+1)Q_{i}$  and  $u_{t_0}(i-1)+1=u_{t_0-1}(i-1)$, we have $\epsilon_{i,t}=1$ for all $t$ and 
   \[( u_0(i),\dots ,u_r(i))= ( u_0(i-1),\dots ,u_r(i-1)),\  \ x_{i,t_0}=Q_i.\]
 \item  If $\epsilon _i=1$,  $u_{t_0}(i-1)Q_{i-1}+x_i\equiv (u_{t_0}(i-1)+1)Q_{i}$ and  $u_{t_0}(i-1)+1<u_{t_0-1}(i-1)$, then $\epsilon_{i,t}=1$ for all $t\not= t_0$,  $\epsilon_{i,t_0}=0$ 
  \[( u_0(i),\dots , u_{t_0-1}(i),u_{t_0}(i), u_{t_0+1}(i) \dots,u_r(i))=\]
  \[= ( u_0(i-1),\dots ,u_{t_0-1}(i-1),u_{t_0}(i-1)+1, u_{t_0+1}(i-1),\dots ,u_r(i-1)). \]  
\item  If $\epsilon _i=1$ and  $u_t(i-1)Q_{i-1}+x_i$ is not equivalent to $( u_t(i-1)+1)Q_i$ for any $t$,
 then $( u_0(i),\dots ,u_r(i))= ( u_0(i-1),\dots ,u_r(i-1))$ and the inequalities are satisfied. 
  \end{enumerate}
  
  Note that case (a) can be seen as a special case of (e) when $x_i=Q_{i-1}$ while case (c) can be seen as a special case of (e) when $x_{i,t_0}=Q_i$.
\end{proof}
 
 \begin{definition} \label{defesp} Let the  $ u_t(i)$ be defined as in Lemma  \ref{ordnalTr} , we say that $x_i$  is  $t_0$-special if  $u_{t_0}(i-1)Q_{i-1}+x_i\equiv (u_{t_0}(i-1)+1)Q_{i}$. 
 We then write $t(i)=t_0$.
 If $x_i$ is $t$-special for some $t$, we say that $x_i$ is  special. If it is not special, we say it is generic.
 \end{definition}
 
 For easy future reference, we list the values of the vanishing at $Q_i$ depending on the data on the corresponding loop:
\begin{corollary} \label{sumanul} The integers $u_t(i)$ defined in   Lemma  \ref{ordnalTr} satisfy  
\begin{enumerate}[(a)]
\item  If $\epsilon _i=0$ and $u_r(i-1)>0$, then  $u_t(i)=u_t(i-1)-1, t=0,\dots,r$.  
\item If $\epsilon _i=0$ and $u_r(i-1)=0$, then $u_t(i)=u_t(i-1)-1, t=0,\dots,r-1, u_r(i)=u_r(i-1)$.
\item If $x_i$ is  $t_0$-special and  $u_{t_0}(i-1)+1=u_{t_0-1}(i-1)$, then  $u_t(i)=u_t(i-1), t=0,\dots,r$.
\item If $\epsilon _i=1$,  $x_i$ is  $t_0$-special  and  $u_{t_0}(i-1)+1< u_{t_0-1}(i-1)$, then  $u_t(i)=u_t(i-1), t\not= t_0, u_{t_0}(i)= u_{t_0}(i-1)+1$.
\item If $\epsilon _i=1$ and $x_i$ generic, then  $u_t(i)=u_t(i-1), t=0,\dots,r$.  
\end{enumerate}
\end{corollary}

We now show the converse of Lemma \ref{ordnalTr} namely

\begin{lemma} \label{ordnalTrinv} Let $\Gamma $ be a general chain  of $g$ loops and $D$ a divisor  on $\Gamma $ 
such that for every  $ k=1,\dots ,g, \ \  \ t=0,\dots , r ,\ \ \  i=0,\dots, g$
there exist indices 
 $ \epsilon_{k},   \ \epsilon_{k,t}\in \{ 0, 1 \}$,
 points in the loops  
 $x_{k}\in \Gamma_k-\{ Q_{k-1}\}, \    x_{k,t}\in \Gamma_k$ and integers 
$u_t(i), \  u_0(i)>\dots>u_r(i)\ge 0 $
such that  
\[ D\equiv tQ_{0}+\sum _{k\le i}\epsilon_{k,t}x_{k,t}+u_t(i)Q_i+\sum _{k>i}\epsilon_{k}x_{k}.\]
Then $D$ has rank at least  $r$.
\end{lemma}

\begin{proof}
In order to show that $D$ has rank $r$, it suffices to see that for any divisor $D'$  of degree at most $r$ with support at $Q_0,\dots, Q_g$,
$D-D'$ is equivalent to an effective divisor (see Theorem 1.6 in \cite{Luo}).
Write $D'=a_0Q_0+\dots+a_gQ_g$. Recall that $D$ is equivalent to  
$a_0Q_0+\epsilon _{1,a_0} x_{1,a_0}+ u_{a_0}(1) Q_1+\sum_{i\ge 2}\epsilon_ix_i$. Then, $D-D'$ is equivalent to
\[ \epsilon _{1,a_0} x_{1,a_0}+ (u_{a_0}(1)-a_1) Q_1+\sum_{i\ge 2}\epsilon_ix_i-\sum _{j\ge2 }a_jQ_j.\]

As $u_0(1)>u_1(1)>\dots >u_r(1)$, $u_{a_0}(1)-a_1\ge u_{a_0+a_1}(1)$.
So, it suffices to check that  
\[  \epsilon _{1,a_0} x_{1,a_0}+ (u_{a_0+a_1}(1)) Q_1+\sum_{i\ge 2}\epsilon_ix_i-\sum _{j\ge2 }a_jQ_j\] 
is effective. This divisor is equivalent to 
\[ \epsilon _{1,a_0} x_{1,a_0}+\epsilon _{2,a_0+a_1} x_{2,a_0+a_1}+  (u_{a_0+a_1}(2)-a_2)Q_2+\sum_{i\ge 3}\epsilon_ix_i-\sum _{j\ge2 }a_jQ_j.\] 

As $u_{a_0+a_1}(2)-a_2\ge u_{a_0+a_1+a_2}(2)$, it suffices to check that  
\[  (u_{a_0+a_1+a_2}(2)) Q_2+\sum_{i\ge 3}\epsilon_ix_i-\sum _{j\ge3 }a_jQ_j\] 
is effective. 
Repeating the argument above $g-1$ times, it will suffice to show that 
\[  \sum_{i=1,\dots,g} \epsilon _{i,a_0+\dots +a_{i-1}} x_{i,a_0+\dots +a_{i-1}}+ (u_{a_0+a_1+\dots +a_{g-1}} (g)-a_g)Q_g\]
is effective. As  $u_{a_0+a_1+\dots +a_{g-1}(g)}-a_g\ge u_{a_0+a_1+\dots +a_{g}}(g)$ and by assumption $a_0+a_1+\dots +a_{g}\le r$,
 then $ u_{a_0+a_1+\dots +a_{g}}(g)\ge u_r(g)\ge 0$ therefore $u_{a_0+a_1+\dots +a_g} (g)$ is well defined and greater than or equal to $0$.
\end{proof}

 %TEOREMA TROPICAL
 \begin{theorem} \label{espaiBNTrop} The Brill-Noether locus  of degree $d$ and rank $r=k-1$ on a general chain of $g$ loops
is a union of  
\[ {g\choose {\rho}} c(k, \bar k)\]
 products of $\rho$ loops corresponding to the $c(k, \bar k)$ fillings of the $k\times \bar k$ Young diagram with $g-\rho=k\bar k$ numbers from the set $1,2,\dots, g$.
The loops appearing in the  product  are the ones whose indices do not appear in the corresponding tableau.
\end{theorem}
\begin{proof} 
Recall that $D$ is equivalent to a divisor of the form $uQ_0+\sum \epsilon _ix_i$ where $u=u_r(0)$.
The orders of vanishing at $Q_0$ were defined as  $( u_0(0),\dots ,u_r(0))=(u-0,\dots, u-r)$ . 
Hence, $\sum _tu_t(0)=(r+1)u-(1+\dots+r)$ .

As the divisor $D$ has degree $d$ and is equivalent to $uQ_0+\sum \epsilon _ix_i$, 
\[ u+\sum_{i=1}^g \epsilon _i=d.\]
So $\sum_{i=1}^g \epsilon _i=d-u$ and there are  $g-d+u$ loops $\Gamma _i$ where the $\epsilon _i=0$.
 Write $\alpha $ for the number of loops where $\epsilon _i=1$ and $x_i$ is generic.
There remain  $d-u-\alpha$ loops where $\epsilon _i=1$ and the $x_i$ is special. 
It follows from Corollary \ref{sumanul}    that 
\[   \sum _{t=0}^ru_t(g)\le  \sum _{t=0}^ru_t(0)-r(g-d+u)+d-u-\alpha \]
\[=\ (r+1)u-(1+\dots+r) -r(g-d+u)+d-u-\alpha \]
with equality when(with the notation in Corollary \ref{sumanul} ) the loops with $\epsilon _i=0$ correspond to case (b) and those with $\epsilon _i=1, x_i$ special correspond to case (d).

On the other hand, as $u_0(g)>\dots >u_r(g)$, the orders of vanishing at $Q_g$ are at least 
$r,\dots,0$. Hence,  
\[r+\dots +1\le  \sum _{t=0}^ru_t(g).\]
 The two inequalities together give
 \[r+\dots +1\le   (r+1)u-(1+\dots+r) -r(g-d+u)+d-u-\alpha,\]
 which gives rise to 
 \[  \alpha\le  (r+1)u-(1+\dots+r) -r(g-d+u)+d-u-(r+\dots+1)=\rho.\]
  Equality in the above inequality  is achieved when the vanishing at both $Q_0, Q_g$ are $(r,\dots, 0)$ and on intermediate components the $\epsilon _i, x_i$
  correspond to choices (with notations as in Corollary \ref{sumanul} ) of type (b),(d), (e).
In situation (b), $\epsilon _i=0$ and there are no further choices to make.
In situation (d), $\epsilon _i=1$, there is a $t_0=t(i)$ such that   $u_{t(i)}(i-1)+1<u_{t(i)-1}(i-1)$.
The $x_i$ is determined by the $ u_{t(i)}(i-1)$, so the only choice is that of the index $t(i)$.
There are no restrictions on when to make a choice of type (e) and then on how to choose $x_i$.
     As $\alpha $ gives the number of loops on which the point is free to vary, optimal choices as in (b), (d) and (e) give rise to a product of $\rho$ loops.
      On the other hand, we pointed out that cases (a) and (c) can be seen as limiting cases of (e).
      Therefore, our loci are products of $\rho $ loops.
 
As $u_r(g)=0$ and in cases (b),  (d), (e) , $ u_r(i)\ge u_r(i-1)$ it follows that  $u_r(i)=0, i=0,\dots, g$ for a generic point on each such loop.

     As $u_s(0)-u_{s-1}(0)=1$ for all $s$, and a choice of type (d) requires $u_{t(i)}(i-1)+1<u_{t(i)-1}(i-1)$, 
     we can only choose $t(i)$ in a type (d) choice for the $n^{th}$ time if each of $0,\dots, t(i)-1$ have already been chosen   at least $n$ times. 
     Similarly, a choice of type (b) can only be made for the $n^{th } $ time if choices of type (d) have been made at least $n$ times for each of the vanishings 
     $0,\dots, r-1$. 
    
    Now construct a Young tableau associated to a component as follows. We number the columns of the tableau from $0 $ to $r$.
     The component determines $\rho$ loops where the $x_i$ will be generic. These $\rho$ loops 
    can be any of  the $g$ loops of $\Gamma$ .    Assign the indices  of the remaining loops successively to one spot of the tableau.
     An index $i$ will be placed in the first empty spot in the column $t(i)$ for a choice of type (d) corresponding to the vanishing $u_{t(i)}$.
    An index $i$ will be placed in the first empty spot in the column $r$ if it corresponds to a choice of type (b).
    By construction, the filling in the columns increase as you go down. Our arguments show that the fillings increase as you move right on a row.
    As in the limit linear series case, we need to show that each column has height $\bar k$.
   Note that $u_t(0)=u-t, u_t(g)=r-t, u+\sum \epsilon_i=d$. 
   Moreover, $u_t(i)=u_t(i-1)-1$ if $\epsilon _i=0$. If  $\epsilon _i=1$,  $u_t(i)=u_t(i-1)$ if $i$ is not $t$-special while $u_t(i)=u_t(i-1)+1$  if $i$ is $t$-special.
   Therefore, $r-t=u_t(g)=u-t-(g-\sum \epsilon _i)+\alpha_t$ with $\alpha _t$ the height of column $t$. It follows that $\alpha _t=\bar k$ for all $t$.
    
    Conversely, if we start with a Young tableau, we can construct a component of the Brill-Noether locus
    as the product of the loops whose indices do not appear in the tableau. 
    If $i$ appears in column $t_0$, write $t(i)=t_0$. Denote by $\beta_{i,t}=\sum _{\{ j\le i\  \} }\delta_{t(j),t}$.
Before defining the divisor corresponding to a point in the component, we need  to say  what we want as the $\epsilon _i$ and the vanishing at the $Q_i$.
Start with $(u_0(0),\dots ,u_r(0))=(r,\dots, 0)$. 
If an index $i$ does not appear on the tableau, take $\epsilon _i=1$ and indices 
    \[(u_0(i),\dots ,u_r(i))=(u_0(i-1),\dots ,u_r(i-1)).\]
      If  $t(i)<r$, take $\epsilon _i=1$ and indices  
     \[(u_0(i),u_1(i),\dots, u_{t(i)}(i),\dots ,u_r(i))=(u_0(i-1),u_1(i-1),\dots ,u_{t(i)}(i-1)+1,\dots ,u_r(i-1)).\]
    If $t(i)=r$,  take $\epsilon _i=0$ 
    \[(u_0(i),u_1(i),\dots ,u_r(i))=(u_0(i-1)-1,u_1(i-1)-1,\dots ,u_{r-1}(i-1)-1,u_r(i-1))\] 
     Note that with this construction, $u_r(i)=0$ for all $i$. 
     
     Then, $u_{t(i)}(i)=u_{t(i)}(i-1)+1 $ if $t(i)=s<r$, $u_{s}(i)=u_{s}(i-1)-1 $ if   $\epsilon _i=0$ and otherwise $u_{s}(i)=u_{s}(i-1)$.
     Therefore, 
     \begin{equation}\label{ordantrop} u_{s}(i)=r-s+\sum_{j\le i}\delta_{s,t(j)}-\sum_{j\le i}\delta_{r,t(j)}=r-s+\beta_{i,s}-\beta_{i,r}\end{equation}
     In particular,  $u_{t(i)}(i) =r-t(i)+\beta_{i,t(i)}-\beta_{i,r}$. 
     As $x_i$ is the unique point such that $u_{t(i)}(i-1)Q_{i-1}+x_i\equiv (u_{t(i)}(i-1)+1)Q_{i}$, we have
           \begin{equation}\label{divtrop} \ (r-t(i)+\beta_{i,t(i)}-\beta_{i,r}-1 )Q_{i-1}+x_i\equiv (r-t(i)+\beta_{i,t(i)}-\beta_{i,r} )Q_{i}.\end{equation}
    For the components whose indices do not appear in the tableau, choose a generic point $x_i$. 
    The generic divisor corresponding to the tableau is then of the form 
    $rQ_0+\sum \epsilon_ix_i$.
\end{proof}

%SERIES POSITIVES
\section{Effective limit linear series}\label{serpos}

We mentioned that the definition of limit linear series comes from concentrating all of the degree of a line bundle successively on each of the components
of a curve of compact type.
The goal of this section is to show that for refined limit linear series, one can concentrate most of the degree and all of the sections on one component while allowing the line bundle to 
still be effective on the remaining components:

\begin{proposition}\label{spdsl}
Assume that $C$ is a curve of compact type with irreducible components $C_j, j=1,\dots, M$. 
Let  $\{ L_j, V_j\subset H^0(C_j, L_j), j=1, \dots , M\}$ be the data of a limit linear series of degree $d$ and dimension $r$ on $C$.
Choose a component $C_{i}$ of $C$. For each $C_j$, let $P_{j,1}, \dots, P_{j,k_j}$  be the set of nodes in $C_j$,   
$X_{j, 1},  \dots, X_{j,k_j}$ the connected components of  $C-C_j$ and  
\[ u_0(j,l)> \dots >u_r(j,l)\ge 0, j=1,\dots, M, l=1,\dots, k_j,\]
 the orders of vanishing of the sections of  $V_j$ at $P_{j, l}$.
If $j\not= i$, let $X_{j, l(j,i)}$be the connected component of $C-C_j$ whose closure  contains $C_i$.

Define a line bundle on $C_j$ by
\begin{align*}
L_{j,i}&=L_j(-u_0(j, l(j,i))P_{j, l(j,i)}-\sum_{l\not= l(j,i)} u_r(j, l)P_{j,l})\\ 
L_{j,j}&=L_j(-\sum_{l=1,\dots, k_j} u_r(j, l)P_{j,l})
\end{align*}
and let $L^i$ be the line bundle obtained by gluing the $L_{j,i}$. 
Note that for $j=i$, no component $X_{j,l}$ contains $C_i$, so the second equation is compatible with the first with the understanding that $l(i,i)$ does not exist.
Then:
\begin{enumerate}
\item The line bundle $L^i$ has degree $d$ on $C$.
\item The restriction of $L^i$ to $C_i$ has a space of sections of dimension $r+1$ that correspond naturally with the sections in $V_i$.
\item The restriction of $L^i$ to $C_j$ has one section.
\end{enumerate}
\end{proposition}

\begin{proof} 
Note that, because the curve $C$ is of compact type, a line bundle on $C$ is completely determined by its restriction to each component. 
So the line bundle $L^i$ on $C$ is well-defined. 

By definition, the restriction $L_{j,i}$ of $L^i$ to $C_j$ is the subsheaf  of sections of $V_j$ generated by those sections with the highest order of vanishing at the node 
closer to $C_i$ and the lowest order of vanishing at the nodes that are further away from $C_i$. 
On the component $C_i$, we look at sections with the lowest order of vanishing at all nodes,
as none of the closures of the irreducible components of $C-C_i$ contains  $C_i$.
 
We now prove our claims:
\begin{enumerate}
\item The degree of a line bundle on a reducible curve is the sum of the degrees of the restriction to each component:
 \[ \deg L^i=\sum_{j=1,\dots, M} \deg L_{j,i}=\sum_{j=1,\dots, M}(d-u_0(j, l(j,i))-\sum_{l\not= l(j,i)} u_r(j, l)).\]
 This sum is ordered with respect to the components $C_j$ of $C$. We can reorder it instead  with respect to  the nodes $P_{\alpha}$ of $C$.
   Every node $P_{\alpha}, \alpha =1,\dots, M-1$, is the  
 intersection of two irreducible components $C_{j_1(\alpha )}, C_{j_2(\alpha )} $ of $C$.
  We choose the indices so that  $C_{j_1(\alpha )}$  is on the same connected component of $C-P_{\alpha}$ as $C_i$ (possibly $C_{j_1(\alpha )}=C_i$)
   and $C_{j_2(\alpha )}$  is not on the same connected component of $C-P_{\alpha}$ as $C_i$.
  Then,  for $C_{j_1(\alpha )}$, either $C_{j_1(\alpha )}=C_i$ or $P_{\alpha}$ is a node that is far from $C_i$ (meaning $P_{\alpha}=P_{j_1(\alpha), l_k}, \  l_k \not= l(j_1(\alpha),i)$). 
  In either case, we are using the vanishing $u_r$ in the definition of $L_{j,i}$.
  For  $C_{j_2(\alpha )}$, $P_{\alpha}$ is a node that is close to $C_i$ (meaning $P_{\alpha}=P_{j_1(\alpha),  l(j_2(\alpha),i)}$).  
   We  rewrite the  equation for the degree of $L^i$ as  
 \[ \deg L^i= d+\sum _{\alpha =1,\dots, M-1}(d-u_r(j_1(\alpha), l(j_1(\alpha, i)))-u_0(j_2(\alpha) ,l_k) ).\]
 If the limit linear series is generic and therefore refined, $u_r(j_1(\alpha ), l(j_1(\alpha), i))+u_0(j_2(\alpha) , l_k)  =d$. Then $ \deg L^i= d$, as claimed.
 \item The sections of $V_i$ vanish at every node $P_{i,l}$ with vanishing at least $u_r(i,l)$. Therefore, the space of sections of 
 $L^i$ restricted to $C_i$ contains all the sections in $V_i$ when considered as sections of $L^i_{|C_i}$ and $V_i$ is a space of dimension $r+1$ by assumption.
 \item On  a component $C_j, j\not=i$, we are considering sections that vanish at one of the nodes with highest order of vanishing. There is one such section on $V_j$ and it vanishes at all other nodes with at least the minimum vanishing. So this section 
 survives in the restriction of $L^i$ to $C_j$.
 \end{enumerate}
 \end{proof}

%DEFINICIO DE SERIES POSITIVES I RELACIO AMB SERIES LIMITS NORMALS
The data we introduced in Proposition \ref{spdsl} of the line bundles $L^i$ defined on the whole reducible curve $C$
is redundant. As in the case of limit linear series, we could minimize the data by considering  only the restrictions of the $L^i$ to $C_i$ and the corresponding space of sections on the components $C_i$ only. We give here a definition and we show that effective linear series are equivalent to the Eisenbud-Harris limit linear series.

\begin{definition} \label{defeff}An \emph{effective linear series} of degree $d$ and dimension $r$ on a curve of compact type $C$ with components $C_i, i=1,\dots, M$, and nodes $P_{\alpha}, \alpha=1,\dots , M-1$, consists of the following data:
\begin{enumerate}
\item A line bundle $L_{i,i}$ of degree $d_i$ on $C_i,\  i=1,\dots ,M$.
\item A space of sections $W_{i}$ of dimension $r+1$ of $L_{i,i}$.
\item For each node $P_{\alpha}$ obtained as the intersection of two irreducible components 
$C_{j_1(\alpha )}, C_{j_2(\alpha )} $ of $C$, an integer $a_{\alpha}, r\le a_{ \alpha} \le d_{j_i(\alpha)}$
\end{enumerate}
These data should satisfy the conditions:
\begin{enumerate}[(a)]
\item $\sum_{i=1,\dots,M}d_i-\sum _{\alpha=1,\dots, M-1}a_{\alpha}=d$.
\item For a node $P_{\alpha}$, consider the orders of vanishing  of the sections of $W_{j_1(\alpha)}$  at the node
(resp the orders of vanishing of the sections of $W_{j_2(\alpha)}$)
\[w_0(j_1(\alpha), \alpha)>\dots>w_r(j_1(\alpha), \alpha)\ \ \ \ \ \ \ \ \ w_0(j_2(\alpha), \alpha)>\dots> w_r(j_2(\alpha), \alpha) \]
Then, $w_t(j_1(\alpha), \alpha)+w_{r-t}(j_2(\alpha), \alpha)\ge a_{\alpha}, t=0,\dots,r$.
\item For each component $C_j$ and every node $P_{\alpha}$ on $C_j$, $W_j(-a_{\alpha}P_{\alpha})$ has at least one section
\end{enumerate}
The series will be called refined when there is an equality in the last condition in (b) for all nodes and all $t$.
\end{definition}

\begin{proposition}
The data of a refined limit linear series and of a refined  effective linear series are equivalent.
\end{proposition}
\begin{proof}
A limit linear series is defined in terms of line bundles on each of the components of a reducible curve and spaces of sections on these individual components.
In Proposition \ref{spdsl}, we saw how  a limit linear series gives rise to line bundles on the whole  curve and spaces of sections of these line bundles.
 Using that construction and with the notations there, we take then $L_{i,i}$ as defined on that proposition, namely $L_{i,i}=L_i(-\sum_{l=1,\dots, k_i} u_r(i, l)P_{i,l})$.
 This line bundle has degree $d_i=d-\sum_{l=0,\dots, k_i} u_r(i,l)$.
 
 If $P_{\alpha}$ is the node formed as the intersection of $C_{j_1(\alpha)}$ and $C_{j_2(\alpha)}$, define 
\[a_{\alpha}=d-u_r(j_1(\alpha), {\alpha})-u_r(j_2(\alpha), {\alpha}).\]

From the conditions on vanishing for a  refined limit linear series, $u_r(j_1(\alpha), {\alpha})+u_0(j_2(\alpha), {\alpha})= d$. Hence
\[ a_{\alpha}=d-u_r(j_1(\alpha), {\alpha})-u_r(j_2(\alpha), {\alpha})= u_0(j_2(\alpha), {\alpha})-u_r(j_2(\alpha), {\alpha})\ge r.\]
Condition (a) for an effective series follows from the definitions.

 As all the sections of $V_i$ vanish at $P_l$ with order at least $u_r(i,l)$, the  space 
 \begin{equation}\label{W_i}   W_i=V_i(-\sum_{l=0,\dots, k_i} u_r(i,l)P_{i,l}).\end{equation}
 is a space of sections of $L_{i,i}$ and still has dimension $r+1$.
Let $w_t(i,l)$ be the order of vanishing of the sections of $W_i$ at $P_l$, that is
\[ w_t(i,l)=u_t(i,l)-u_r(i,l)\]
The condition $u_0(i,l)>\dots>u_r(i,l)$ then implies $w_0(i,l)>\dots>w_{r-1}(i,l)>w_r(i,l)=0$ which implies the first condition in Definition \ref{defeff} part (b).

 As $u_t(j_1(\alpha),\alpha)+u_{r-t}(j_2(\alpha), \alpha)= d$,
  \begin{align*}
w_t(j_1(\alpha), \alpha)+w_{r-t}(j_2(\alpha), \alpha)&=u_t(j_1(\alpha), \alpha)-u_r(j_1(\alpha), \alpha)+u_{r-t}(j_2(\alpha), \alpha)-u_r(j_2(\alpha), \alpha)\\
&= d-u_r(j_1(\alpha), \alpha)-u_r(j_2(\alpha), \alpha)=a_{\alpha}.
\end{align*}
 which proves the second part of condition (b) for refined series. 
  
  Note now that if the irreducible components of $C$ containing the node $P_{\alpha}$ are $C_{j_1}, C_{j_2}$ with $P_{\alpha}=P_{j_1, l_1}=P_{j_2, l_2}$, 
   \begin{align*}
W_{j_1}(-a_{\alpha}P_{\alpha})&\supseteq  V_{j_1}(-\sum_{m=0,\dots, k_{j_1}} u_r(j_1,m)P_{j_1,m}- (d-u_r(j_1,l_1)-u_r(j_2,l_2))P_{\alpha})\\
&= V_{j_1}(-\sum_{m\not= l_1} u_r(j_1,m)P_{j_1,m}- (d-u_r(j_2,l_2))P_{\alpha})\\&\supseteq  V_{j_1}(-\sum_{m\not= l_1} u_r(j_1,m)P_{j_1,m}-u_0(j_1,l_1)P_{\alpha}),
  \end{align*}
  where we used that $u_0(j_1, l_1)+u_r(j_2, l_2)\ge d$.
  By definition of the orders of vanishing, this latter space has a section. In particular, this implies that $ a_{ \alpha} \le d_{j_i(\alpha)}$.
  
Conversely, an effective refined linear series $(L_{i,i}, W_i, a_{\alpha}), i=1,\dots, M, \alpha=1, \dots, M-1$, determines  a limit linear series
$(L_i, V_i), i=1,\dots, M$, as follows: given a component 
$C_j$, let $P_{j,1}, \dots, P_{j,k_j}$  be the set of nodes in $C_j$ and 
$X_{j, 1},  \dots, X_{j,k_j}$ the corresponding connected components of  $C-C_j$.

Define 
\[d'_{j,l}=\sum _{C_m\in X_{j,l}}d_m-\sum _{P_{\alpha}\in  X_{j,l}}a_{\alpha},\ \  L_j=L_{j,j}(\sum_l d'_{j,l}P_{j,l}).\]
The condition $a_{\alpha}\le d_{j_i(\alpha)}$ in (iii) guarantees that $d'_{j,l}\ge 0$. Then, 
\[ \deg L_j=d_j+\sum_l\sum _{C_m\in X_{k,l}}d_m-\sum _{P_{\alpha}\in  X_{k,l}}a_{\alpha}=d_j+\sum_{l\not= j}d_m-\sum _{\alpha}a_{\alpha}=d.\]

Define 
\[ V_j=W_j(\sum_l d'_{j,l}P_{j,l}).\] 
What we mean here is that we take the same spaces of sections $W_j$ with fixed points of multiplicities $d'_{j,l}$ at $P_{j,l}$.
Then using the second part of condition (b) in \ref{defeff}
\[  a_{\alpha}+d'_{j_1,\alpha}+d'_{j_2,\alpha}= u_t(j_1(\alpha), \alpha)+u_{r-t}(j_2(\alpha), \alpha)\]
\[  = a_{\alpha}+\sum _{C_m\in X_{j_1,\alpha}}d_m-\sum _{P_{\beta}\in X_{j_1, \alpha}}a_{\beta}
+\sum _{C_m\in X_{j_2,\alpha}}d_m-\sum _{P_{\beta}\in X_{j_2, \alpha}}a_{\beta}=\sum_{i=1,\dots, M}d_i-\sum_{\beta =1,\dots, M-1} a_{\beta}=d \] 
where the last equality comes from condition (a) in \ref{defeff}.
This concludes the proof of the fact that $(L_i, V_i)$ gives the data of a limit linear series.
\end{proof}   

Recall that Young tableaux of dimension $(r+1)(g-d+r)$ filled with integers among $1,\dots, g$
correspond to generic component of the image in the Jacobian of the set of limit linear series of degree $d$ and dimension $r$ on a   general chain of elliptic curves.
If an index $i$ appears in the tableau on column $t_0$, we write $t_0=t(i)$. Denote by $\beta_{i,t}=\sum _{\{ j\le i\  \} }\delta_{t(j),t}$.
In particular, $i$ appears in row  $\beta_{i,t(i)}$.

From the correspondence between refined limit linear series and refined effective series, these tableaux correspond also to effective linear series of degree $d$ and dimension $r$.
We describe the correspondence below.
 
%DESCRIPCIO DE SERIE EFECTIVA A CADENA EL.LIPTICA
\begin{lemma} \label{descreff}Let $C$ be a general chain of elliptic curves. 
Given a Young tableau of dimension $(r+1)(g-d+r)$ filled with integers among $1,\dots, g$,
consider a general point of the component of the Brill-Noether locus on the chain corresponding to the tableau.
This point gives rise to a limit linear series. 
The  line bundle $L^1$ defined in \ref{spdsl} from this limit linear series is described as follows:
\[L^1_{|C_i}=\begin{cases}
{\mathcal O}_{C_i}&  \textrm{if } t(i)= r\\
{\mathcal O}_{C_i}(x_i)\ \ \  \  x_i+(r+\beta_{i,t(i)}-t(i)-\beta_{i,r}-1) P_i\equiv  (r+\beta_{i,t(i)}-t(i)-\beta_{i,r} )  Q_i & \textrm{if } t(i)<r\\
{\mathcal O}_{C_i}(x_i)\ \ \  \  x_i \textrm{ generic}  &  \textrm{if } i \textrm{ not in tableau. }   \end{cases}\]

\end{lemma}

\begin{proof} From the correspondence between limit linear series and Tableaux, if an index is not on the tableau, then $L_i$ is a general line bundle of degree $d$.
If $i$  appears in the tableau, then from equation (\ref{fibrat}), 
\[ L_i={\mathcal O}_{C_i}((d-u_{t(i)}(i))P_i+u_{t(i)}(i)Q_i)= \mathcal {O}(( t(i)+i-\beta_{i,t(i)})P_{i}+(d- t(i)-i+\beta_{i,t(i)})Q_{i}).\]
 Using equation (\ref{anulserlim}) the orders of vanishing of $V_i$ are written as   $u_s(i) =d-s-i+\beta_{i,s}$. 
The orders of vanishing $v_s(i)$ of the sections at $P_i$   are given by $v_s(i)=d-u_{r-s}(i)-1, s\not= r-t(i); \ v_{r-t(i)}(i)=d-u_{t(i)}(i)$

From the definition in \ref{spdsl},
 \begin{align*} 
L_{1,1}&=L_1(-u_r(1)Q_1)& ,\\ 
L_{i,1}&=L_i(-(v_0(i)P_i-u_r(i)Q_i))=&\begin{cases}
L_i(-(d-u_r(i)-1)P_i-u_r(i)Q_i),& t(i)\not= r, \\  
 L_i(-(d-u_r(i))P_i-u_r(i)Q_i), & t(i)= r.\end{cases}
 \end{align*}

From $L_i=\mathcal{O}((d-u_{t(i)}(i))P_i+u_{t(i)}(i)Q_i)$, if the index $i$ is on the last column ($t(i)=r$), then $L_{i,1}={\mathcal O}_{C_i}.$

If  $t(i)<r$, substituting the values of $L_i,  u_r(i)$,  we obtain
\[  L_{i,1}=L_i(-(d-u_r(i)-1)P_i-u_r(i)Q_i)=\]
\[\mathcal {O}(( t(i)+i-\beta_{i,t(i)})P_{i}+(d- t(i)-i+\beta_{i,t(i)})Q_{i} )(-(d-(d-r-i+\beta_{i,r}-1)P_i-(d-r-i+\beta_{i,r})Q_i))=\]
\[   = \mathcal{O}((t(i)-\beta_{i,t(i)}-r+\beta_{i,r}+1)P_i+(r-t(i)+\beta_{i,t(i)}-\beta_{i,r})Q_i) \]
As $L_{i,1}$ is a line bundle of degree 1 on an elliptic curve, we have $L_{i,1}={\mathcal O}_{C_i}(x_i)$, where $x_i$ satisfies the condition in the statement.

If the index $i$ does not appear in the tableau, $L_i$ is a general line bundle of degree $d$  therefore 
$L_{i,1}$ is a general line bundle of degree 1 on $C_i$ and we can write 
$L_{i,1}={\mathcal O}_{C_i}(x_i)$ where $x_i$ is a generic point of $C_i$.
\end{proof}

In the previous lemma, we computed the line bundles $L_{i,1}$. We can similarly compute the $L_{i,j}$ for other values of $j$. 
We can also find the spaces of sections
$W_j$ of $L_{j,j}$. From equation (\ref{W_i}), in our situation $W_i=V_i(-u_r(i)Q_i-(d-u_0(i)-1)P_i)$ if $t(i)\not= 0$  or $i$ is not on the tableau 
and $W_i=V_i(-u_r(i)Q_i-(d-u_0(i))P_i)$ if $t(i)= 0$.
It follows that the orders of vanishing of the sections of $W_j$ at $Q_j$ are $u_0(j)-u_r(j),\dots u_{r-1}(j)-u_r(j), u_r(j)-u_r(j)=0$. 
Using the expression in equation (\ref{anulserlim}), $u_s(i) =d-s-i+\beta_{i,s}$, the expression for $w_s$ is given by 
\begin{equation} \label{anuleff}  w_s(i)=r-s+ \beta_{i, s}-\beta_{i,r} \end{equation}

\begin{lemma} \label{descrtrop}Let $C$ be a general chain of loops. 
Given a Young tableau of dimension $(r+1)(g-d+r)$ filled with integers among $1,\dots, g$,
consider a general point of the component of the Brill-Noether locus on the chain corresponding to the tableau.
This point gives a divisor of the form   \[ rP_{1} +\sum _{k\ge1}\epsilon_{k}x_{k}.\] 
\[  \begin{cases}
\epsilon_i=0&  \textrm{if } t(i)=  r\\
\epsilon_i=1  \ \ \ \  x_i+(r+\beta_{i,t(i)}-t(i)-\beta _{r,i}-1) P_{i}\equiv  (r\beta_{i,t(i)}-t(i)-\beta_{i,r} ) Q_i &  \textrm{if } t(i)<r\\
\epsilon_i=1  \ \ \ \  x_i \ \textrm{ generic}  & \textrm{if } i \textrm{ not in tableau.}   \end{cases}\]
\end{lemma}
\begin{proof} This is a recap of Section \ref{sectrop}. In particular, the description of the divisor follows from equation (\ref{divtrop}).
\end{proof}

\begin{remark}\label{Rem}  The statements of Lemmas \ref{descreff} and \ref{descrtrop} give a direct analogy between effective line bundles 
in the Brill-Noether locus for a chain of elliptic curves and divisor on tropical chains of loops.
We point out also that the orders of vanishing of these divisors or line bundles agree (see equations (\ref{ordantrop}) and (\ref{anuleff})).
As the proof of the Brill-Noether theorem relies on the positivity of these vanishing, the tropical or limit linear series proofs run in parallel.
\end{remark}

 %EJEMPLO
\begin{example} In this example, we exhibit the analogy and correspondence between the theory of effective limit linear series on a chain of elliptic curves 
and the theory of divisors on a tropical chain of loops. We look at the case when $g=6$, $d=6$, and $r=2$, and therefore $\rho=0$.

Let $C$ be a chain of elliptic curves of genus $6$ as in Definition \ref{cce}. Let us consider the example of Eisenbud-Harris limit linear series on this chain associated to the Young tableau \begin{center}
\medskip

\begin{center}
    \begin{tabular}{| c | c | c | }
    \hline
    1 & 2 &4\\
     \hline
3 & 5&6\\
 \hline
       \end{tabular}
\end{center}
\medskip
\end{center}
\medskip
The data of this limit is summarized in the table in Figure 1: the $i$th row corresponds to the irreducible component $C_i$, we give the corresponding degree 6 line bundle on $C_i$ followed by the orders of vanishing of the linearly independent sections at $P_i$ and $Q_i$, respectively. 

\bigskip
\begin{minipage}[t]{\dimexpr.43\textwidth-.45\columnsep}
\begin{center}

\begin{tikzpicture}[inner sep=0in,outer sep=0in]
\node (s) {\begin{varwidth}{5cm}{
\begin{center}
    \begin{tabular}{lccccc}
    &&&$s^1_i$ & $s^2_i$ & $s^3_i$ \\ \cline{4-6}
    \multirow{2}{*}{$\scriptstyle C_1$\phantom{mmm}} &\multirow{2}{*}{$\mathcal O(6Q_1)$}&&\color{orange}0&1&2 \\ 
    &&&\color{orange}6&4&3 \\  \cline{4-6}
    
        \multirow{2}{*}{$\scriptstyle C_2$}&\multirow{2}{*}{$\mathcal O(2P_2+4Q_2)$}&&0&\color{orange}2&3  \\ 
    &&&5&\color{orange}4&2 \\  \cline{4-6}
    
        \multirow{2}{*}{$\scriptstyle C_3$}&\multirow{2}{*}{$\mathcal O(P_3+5Q_3)$}&&\color{orange}1&2&4  \\ 
    &&&\color{orange}5&3&1 \\  \cline{4-6}
    
        \multirow{2}{*}{$\scriptstyle C_4$}&\multirow{2}{*}{$\mathcal O(5P_4+Q_4)$}&&1&3&\color{orange}5  \\ 
    &&&4&2&\color{orange}1 \\ \cline{4-6}
    
        \multirow{2}{*}{$\scriptstyle C_5$}&\multirow{2}{*}{$\mathcal O(4P_5+2Q_5)$}&&2&\color{orange}4&5  \\ 
    &&&3&\color{orange}2&0 \\  \cline{4-6}
    
        \multirow{2}{*}{$\scriptstyle C_6$}&\multirow{2}{*}{$\mathcal O(6P_6)$} &&3&4&\color{orange}6  \\  
    &&&2&1&\color{orange}0 \\ 

    \end{tabular}
\end{center}}\end{varwidth}};
\draw [black,  thick] plot [smooth, tension=1] coordinates { (-1.45,2.55) (-1.7,2.05) (-1.45,1.55)};
\draw [black,  thick] plot [smooth, tension=1] coordinates { (-1.45,1.65) (-1.7,1.15) (-1.45,0.65)};
\draw [black,  thick] plot [smooth, tension=1] coordinates { (-1.45,0.75) (-1.7,0.25) (-1.45,-0.25)};
\draw [black,  thick] plot [smooth, tension=1] coordinates { (-1.45,-0.15) (-1.7,-0.65) (-1.45,-1.15)};
\draw [black,  thick] plot [smooth, tension=1] coordinates { (-1.45,-1.05) (-1.7,-1.55) (-1.45,-2.05)};
\draw [black,  thick] plot [smooth, tension=1] coordinates { (-1.45,-1.95) (-1.7,-2.45) (-1.45,-2.95)};
\draw (1,-3.75) node {\textsc{Figure 1.} Eisenbud-Harris limit linear series};
\end{tikzpicture}

\end{center}
\end{minipage}% <---------------- Note the use of "%"
\begin{minipage}[t]{\dimexpr.6\textwidth-.0\columnsep}
\begin{center}
\begin{tikzpicture}[inner sep=0in,outer sep=0in]
\node (s) {\begin{varwidth}{5cm}{
\begin{center}
    \begin{tabular}{lccccc}
    &&&$s^1_i$ & $s^2_i$ & $s^3_i$ \\ \cline{4-6}
    \multirow{2}{*}{$\scriptstyle C_1$\phantom{mmm}} &\multirow{2}{*}{$\mathcal O(3Q_1)$}&&\color{cyan}0&1&2 \\ 
    &&&\color{cyan}3&1&0 \\  \cline{4-6}
    
        \multirow{2}{*}{$\scriptstyle C_2$}&\multirow{2}{*}{$\mathcal O(2P_2+2Q_2)$}&&0&\color{cyan}2&3  \\ 
    &&&3&\color{cyan}2&0 \\  \cline{4-6}
    
        \multirow{2}{*}{$\scriptstyle C_3$}&\multirow{2}{*}{$\mathcal O(4Q_3)$}&&\color{cyan}0&1&3  \\ 
    &&&\color{cyan}4&2&0 \\  \cline{4-6}
    
        \multirow{2}{*}{$\scriptstyle C_4$}&\multirow{2}{*}{$\mathcal O(4P_4)$}&&0&2&\color{cyan}4  \\ 
    &&&3&1&\color{cyan}0 \\ \cline{4-6}
    
        \multirow{2}{*}{$\scriptstyle C_5$}&\multirow{2}{*}{$\mathcal O(2P_5+2Q_5)$}&&0&\color{cyan}2&3  \\ 
    &&&3&\color{cyan}2&0 \\  \cline{4-6}
    
        \multirow{2}{*}{$\scriptstyle C_6$}&\multirow{2}{*}{$\mathcal O(3P_6)$} &&0&1&\color{cyan}3  \\  
    &&&2&1&\color{cyan}0 \\ 

    \end{tabular}
\end{center}}\end{varwidth}};
\draw [black,  thick] plot [smooth, tension=1] coordinates { (-1.45,2.55) (-1.7,2.05) (-1.45,1.55)};
\draw [black,  thick] plot [smooth, tension=1] coordinates { (-1.45,1.65) (-1.7,1.15) (-1.45,0.65)};
\draw [black,  thick] plot [smooth, tension=1] coordinates { (-1.45,0.75) (-1.7,0.25) (-1.45,-0.25)};
\draw [black,  thick] plot [smooth, tension=1] coordinates { (-1.45,-0.15) (-1.7,-0.65) (-1.45,-1.15)};
\draw [black,  thick] plot [smooth, tension=1] coordinates { (-1.45,-1.05) (-1.7,-1.55) (-1.45,-2.05)};
\draw [black,  thick] plot [smooth, tension=1] coordinates { (-1.45,-1.95) (-1.7,-2.45) (-1.45,-2.95)};
\draw (1,-3.75) node {\textsc{Figure 2.} Effective limit linear series};
\end{tikzpicture}
\end{center}
\end{minipage}

\bigskip

Let us now construct the corresponding \emph{effective} limit linear series on the chain of elliptic curves $C$ using the results in Section 3: 
\begin{itemize}
\item for $i=1, j=1$, we have 
$
u_0(1,1)=2, u_1(1,1)=1, u_2(1,1)=0\mbox{ and }
u_0(1,2)=6, u_1(1,2)=4, u_2(1,2)=3,
$
and then
$$
	L_{1,1}=L_1(-u_2(1,1)P_1-u_2(1,2)Q_1)=L_1(-0P_1-3Q_1)=\mathcal O_{C_1}(3Q_1);
$$
\item for  $i=1, j=2$, we have 
$
u_0(2,1)=3, u_1(2,1)=2, u_2(2,1)=0\mbox{ and }
u_0(2,2)=5, u_1(2,2)=4, u_2(2,2)=2,
$
and then
$$
L_{2,1}=L_2(-u_0(2,1)P_2-u_2(2,2)Q_2)=L_2(-3P_2-2Q_2)=\mathcal O_{C_2}(2Q_2-P_2);
$$
\item for $i=1,j=3$, we have 
$
u_0(3,1)=4, u_1(3,1)=2, u_2(3,1)=1\mbox{ and }
u_0(3,2)=5, u_1(3,2)=3, u_2(3,2)=1,
$
and then
$$
L_{3,1}=L_3(-u_0(3,1)P_3-u_2(3,2)Q_3)=L_3(-4P_3-Q_3)=\mathcal O_{C_3}(4Q_3-3P_3).
$$
\end{itemize}

Similar computations give 
\begin{itemize}
\item $L_{4,1}=L_4(-u_0(4,1)P_4-u_2(4,2)Q_4)=L_4(-5P_4-Q_4)=\mathcal O_{C_4},$
\item $L_{5,1}=L_5(-5P_5)=\mathcal O_{C_5}(2Q_5-P_5),$ and 
\item $L_{6,1}=L_6(-6P_6)=\mathcal O_{C_6}.$
\end{itemize}
\medskip

A complete description of the bundles $L_{j,i}$ is shown on the following table:
\medskip
\begin{center}
    \begin{tabular}{| c | c | c | c | c | c |}
    \hline
    $L_{1,1}$ & $L_{2,1}$ & $L_{3,1}$ & $L_{4,1}$ & $L_{5,1}$ & $L_{6,1}$\\ \hline
    
$    \color{cyan}\mathcal O (3Q_1)$ & $\mathcal O$ & $\mathcal O $ & $\mathcal O$ & $\mathcal O $ & $\mathcal O$ \\ \hline
    
$\mathcal O (2Q_2-P_2)$ & \color{cyan}$\mathcal O(2P_2+2Q_2)$ & $\mathcal O(2P_2-Q_2) $ & $\mathcal O(2P_2-Q_2)$ & $\mathcal O(2P_2-Q_2) $ & $\mathcal O(2P_2-Q_2) $ \\ \hline

$\mathcal O (4Q_3-3P_3)$ & $\mathcal O(4Q_3-3P_3)$ & \color{cyan}$\mathcal O(4Q_3) $ & $\mathcal O$ & $\mathcal O $ & $\mathcal O$ \\ \hline

$\mathcal O$ & $\mathcal O$ & $\mathcal O $ & \color{cyan}$\mathcal O(4P_4)$ & $\mathcal O(3P_4-2Q_4) $ & $\mathcal O(3P_4-2Q_4) $ \\ \hline

$\mathcal O(2Q_5-P_5) $ & $\mathcal O(2Q_5-P_5)$ & $\mathcal O(2Q_5-P_5) $ & $\mathcal O(2Q_5-P_5)$ & \color{cyan}$\mathcal O(2P_5+2Q_5) $ & $\mathcal O(2P_5-Q_5) $ \\ \hline

$\mathcal O $ & $\mathcal O$ & $\mathcal O $ & $\mathcal O$ & $\mathcal O $ & \color{cyan}$\mathcal O (3P_6)$ \\ \hline

    \end{tabular}
\end{center}
\medskip

The data for the \emph{effective} limit linear series is summarized in Figure 2, following the same conventions as in the Eisenbud-Harris limit.

Let now $\Gamma$ be a general chain of $6$ loops. The divisor corresponding to the tableau is of the form 
\[ 2Q_0+x_1+x_2+x_3+x_5\]
where $x_i$ is on the $ i^{th }$ loop. The points $x_i$  satisfy 
\[ {\color{cyan} 2Q_0+x_1\equiv 3Q_1},\ {\color{red} Q_1+x_2\equiv 2Q_2}, \ {\color{cyan} 3Q_2+x_3\equiv 4Q_3}, {\color{red}Q_4+x_5\equiv 2Q_5}.\]
\setcounter{figure}{2}
\begin{figure}[H] \label{Fig:ChainOfLoops}
\begin{tikzpicture}
\draw [ball color=black] (-1.1,0) circle (0.55mm);
\draw [ball color=black] (-.18,-.5) circle (0.15mm);
\draw (-0.5,0) circle (0.6);
\draw [fill=cyan](-0.18,-.5) circle (0.06);
\draw [ball color=black] (0.05,0.2) circle (0.55mm);
\draw  (0.7,1.2) circle (0.05);
\draw (0.7,0.5) circle (0.7);
\draw [fill=red] (0.7,1.2) circle (0.55mm);
\draw [ball color=black] (1.4,0.35) circle (0.55mm);
\draw (2,0.3) circle (0.6);
\draw  [fill=cyan](1.72,-.23) circle (0.06);
\draw [ball color=black] (2.6,0.257) circle (0.55mm);
\draw (3.4,0.3) circle (0.8);
\draw [ball color=black] (4.1,0.7) circle (0.55mm);
\draw (4.5,1) circle (0.5);
\draw [fill=red] (4.7,1.45) circle (0.55mm);
\draw [ball color=black] (5,0.9) circle (0.55mm);
\draw (5.57,0.75) circle (0.6);
\draw [ball color=black] (6.2,0.7) circle (0.55mm);

\node at (-1,-1) {$\scriptsize\mat{\\1}$};
\node[color=cyan, scale=.8] at (-1,-.75) {2};

\node at (0.1,-1) {$\scriptsize\mat{3\\ { }}$};
\node[color=red, scale=.8] at (0.1,-1.2) {1};

\node at (1.4,-1) {$\scriptsize\mat{\\2}$};
\node[color=cyan, scale=.8] at (1.4,-.8) {3};

\node at (2.6,-1) {$\scriptsize\mat{4\\2}$};
\node at (4.1,-1) {$\scriptsize\mat{3\\{ }}$};
\node[color=red, scale=.8] at (4.1,-1.2) {1};

\node at (5,-1) {$\scriptsize\mat{3\\2}$};
\node at (6.15,-1) {$\scriptsize\mat{2\\1}$};

\end{tikzpicture}
\caption{The chain $\Gamma$ with vanishing orders at the nodes.}
\end{figure}
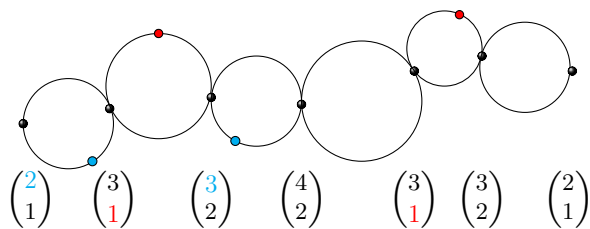

\end{example}

\noindent {\bf Acknowledgements.}
The authors would like to thank the referee for a very careful reading of the manuscript. 
During the summer of 2014. they had fruitful discussions on related topics  with Melody Chan and Nathan Pflueger.   
The first author would also like to thank Sam Payne and Jan Draisma for helping him understand the way they thought about the combinatorics of chip-firing during the BIRS workshop ``Specialization of Linear Series for Algebraic and Tropical Curves'' in April 2014 and the Max-Planck Institut f\"ur Mathematik in Bonn for the wonderful working conditions and stimulating environment over the duration of his visit in the summer of 2014. 
%\bibliography{biblio}

\begin{thebibliography}{ACGH85}


\bibitem[ACGH85]{ACGH}
E.~Arbarello, M.~Cornalba, P.~A.~Griffiths, and J.~Harris.
\newblock {\em Geometry of algebraic curves. {V}ol. {I}}, volume 267 of {\em
  Grundlehren der Mathematischen Wissenschaften [Fundamental Principles of
  Mathematical Sciences]}.
\newblock Springer-Verlag, New York, 1985.

\bibitem[CLT]{CLT}
A.~Castorena, A.~L\'opez Mart\'in, and M.~Teixidor i Bigas.
\newblock {\em Invariants of the Brill-Noether curve.}  To appear in Advances in Geometry.
  \href{http://arxiv.org/abs/1403.7489}{arXiv:1403.7489}

\bibitem[CT]{CT}
A.~Castorena and M.~Teixidor i Bigas.
\newblock {\em Divisorial components of the Petri locus for pencils.} 
\newblock {Jounal Pure Applied Algebra} 212, 2008, 1500-1508.

\bibitem[CDPR12]{CDPR}
F.~Cools, J.~Draisma, S.~Payne, E.~Robeva.
\newblock {\em A tropical proof of the Brill-Noether theorem.} 
\newblock{ Adv. Math. 230, 759--776, 2012.}

\bibitem[CLPT]{CLPT}
M.~Chan, A.~L\'opez Mart\'in, N.~Pflueger, M.~Teixidor i Bigas.
\newblock {\em The genus of the Brill-Noether curve.} 
\href{ http://arxiv.org/abs/1506.00516}{arxiv:1506.00516}

\bibitem[EH86]{EH86}
D.~Eisenbud and J.~Harris.
\newblock {\em Limit linear series: basic theory.}
\newblock { Invent. Math.}, 85(2): 337--371, 1986.

\bibitem[LOTZ]{LOTZ} F. Liu, B.Osserman, M.Teixidor i Bigas, N.Zhang.  Limit linear series and rank of multiplication maps. In preparation.

\bibitem[HM98]{HM98}
J.~Harris and I.~Morrison.
\newblock Moduli of curves.
\newblock {\em Graduate Texts in Mathematics}, 187. Springer-Verlag.

\bibitem[L]{Luo}
Y.~Luo.
\newblock {\em Rank determining sets of metric graphs}  Jour.Comb.Th.A, 118, 2011, 1775-1793.

\bibitem[Oss05]{Oss05}
B.~Osserman.
\newblock {\em Two degeneration techniques for maps of curves.}  Contemp. Math., 388, 
\newblock {\em Snowbird lectures in algebraic geometry},
 137-143, Amer. Math. Soc., Providence, RI, 2005.
 
\bibitem[Oss]{Oss}
B.~Osserman.
\newblock  {\em Limit linear series moduli stacks in higher rank.}
\href{http://arxiv.org/abs/1405.2937}{arXiv:1405.2937}

\bibitem[Pfl]{P}
N.~Pflueger.
\newblock  {\em On linear series with negative Brill-Noether number.}
\href{http://arxiv.org/abs/13115845}{arXiv:13115845}

\bibitem[Tei04]{Tei04}
M. Teixidor i Bigas.
\newblock Rank two vector bundles with canonical determinant.
\newblock {\em Math.Nach.}, 265:100--106, 2004.

\bibitem[Tei05]{Tei05}
M. Teixidor i Bigas.
\newblock Existence of vector bundles of rank two with sections.
\newblock {\em Adv. Geom.}, 5(1):37--47, 2005.

\bibitem[Tei08]{Tei08b}
M. Teixidor i Bigas.
\newblock Existence of coherent systems. {II}.
\newblock {\em Internat. J. Math.}, 19(10):1269--1283, 2008.

\bibitem[Tei08b]{TeiPet}
M. Teixidor i Bigas.
\newblock Petri map for rank two bundles with canonical determinant.
\newblock {\em Compos. Math.}, 144(10):705-720, 2008.

\bibitem[Tei11]{Clay}
M. Teixidor i Bigas.
\newblock Vector bundles on reducible curves and applications.
\newblock {\em in Grassmannians, moduli spaces and vector bundles}, AMS CMI 14,  169--180, 2011.

\bibitem[Wel85]{W}
G.~E.~Welters.
\newblock {\em  A theorem of Gieseker-Petri type for Prym varieties.}
\newblock {Ann. Sci. \'Ecole Norm. Sup.} (4) 18: no. 4, 671--683, 1985.



\end{thebibliography}
\def\cprime{$'$}

\end{document}